\documentclass[11pt,oneside]{article}
\usepackage{supertabular, setspace,hyperref}
\usepackage[utf8]{inputenc}
\usepackage{amsmath}
\usepackage{amsfonts}
\usepackage{amssymb}
\usepackage{amsthm}
\usepackage{mathtools}
\newtheorem{theorem}{Theorem}[section]

\newtheorem{example}[theorem]{Example}

\newtheorem{remark}{\sc Remark}
\newtheorem{lemma}{\sc Lemma}[section]
\newtheorem{corollary}{\sc Corollary}[section]
\newtheorem{definition}{\sc Definition}[section]
\newtheorem{proposition}{\sc Proposition}[section]

\newcommand{\be}{\begin{eqnarray}}
\newcommand{\ee}{\end{eqnarray}}
\newcommand{\Be}{\begin{eqnarray*}}
	\newcommand{\Ee}{\end{eqnarray*}}
\newcommand{\bee}{\begin{equation}}
\newcommand{\eee}{\end{equation}}
\newcommand{\ba}{\begin{array}}
	\newcommand{\ea}{\end{array}}
\newcommand{\bl}{\begin{lemma}}
	\newcommand{\el}{\end{lemma}}
\newcommand{\bd}{\begin{definition}}
	\newcommand{\ed}{\end{definition}}
\newcommand{\bt}{\begin{theorem}}
	\newcommand{\et}{\end{theorem}}
\newcommand{\bp}{\begin{proof}}
	\newcommand{\ep}{\end{proof}}
\newcommand{\bi}{\begin{itemize}}
	\newcommand{\ei}{\end{itemize}}
\newcommand{\br}{\begin{remark}}
	\newcommand{\er}{\end{remark}}
\newcommand{\bc}{\begin{corollary}}
	\newcommand{\ec}{\end{corollary}}
\newcommand{\bex}{\begin{example}}
	\newcommand{\eex}{\end{example}}
\begin{document}
	\date{}
	\title{\textbf{Homogeneous Finsler spaces with some special $(\alpha, \beta)$-metrics}}
	\maketitle
	\begin{center}
		\author{\textbf{Gauree Shanker and Kirandeep Kaur}}
	\end{center}
	\begin{center}
		Department of Mathematics and Statistics\\
		School of Basic and Applied Sciences\\
		Central University of Punjab, Bathinda, Punjab-151001, India\\
		Email:   gshankar@cup.ac.in, kiran5iitd@yahoo.com
	\end{center}
	\begin{center}
		\textbf{Abstract}
	\end{center}
	\begin{small}
			In this paper, first we prove the existence of invariant vector field on a homogeneous Finsler space with infinite series $(\alpha, \beta)$-metric and exponential metric. Next, we deduce an explicit formula for the  the $S$-curvature of homogeneous Finsler space with these metrics. Using this formula, we further derive the formula for mean Berwald curvature of the homogeneous Finsler space with the above mentioned metrics. 
	\end{small}\\
	\textbf{2010 Mathematics Subject Classification:} 22E60, 53C30, 53C60.\\
	\textbf{Keywords and Phrases:} Homogeneous Finsler space, invariant vector field,  infinite series $(\alpha, \beta)$-metric, exponential metric, $S$-curvature, mean Berwald curvature.
	\section{Introduction}
	The main purpose of this paper is to give a formula for   $S$-curvature and mean Berwald curvature of the homogeneous Finsler space with infinite series $(\alpha, \beta)$-metric and also for exponential metric. According to S. S. Chern (\cite{1996 Chern}), Finsler geometry is just the Riemannian geometry without the quadratic restriction. The notion of  $(\alpha, \beta)$-metric in Finsler geometry was introduced by M. Matsumoto in 1972 (\cite{M.Mat1972}). An $(\alpha,\beta)$-metric is a Finsler metric of the form $F= \alpha \phi(s), \ s= \dfrac{\beta}{\alpha}$, where $\alpha= \sqrt{a_{ij}(x)y^iy^j}$ is a Riemannian metric on a connected smooth $n$-manifold $M$ and $\beta= b_i(x) y^i$ is a 1-form on $M$. It is well known fact that  $(\alpha,\beta)$-metrics are the generalizations of the Randers metric introduced by G. Randers in (\cite{Randers}). $(\alpha,\beta)$-metrics have various applications in physics and biology (\cite{AIM}). 	Consider the $r^{\text{th}}$ series $(\alpha,\beta)$-metric:
	$$ F(\alpha, \beta)= \beta \sum_{k=0}^{k= \infty}\left(\dfrac{\alpha}{\beta} \right)^k .$$
	If $r=1$, then it is a Randers metric.\\
	If $r=\infty$, then $$ F= \dfrac{\beta^2}{\beta - \alpha}.  $$ 
	This metric is called an infinite series $(\alpha,\beta)$-metric. Interesting fact about this metric is that, it is the difference of a Randers metric and a Matsumoto metric, and satisfies Shen's lemma (see lemma \ref{existence of metric}).\\
	Some other important class of  $(\alpha,\beta)$-metrics are Randers metric, Kropina metric, Matsumoto metric and exponential metric etc.
		Many authors (\cite{CS}, \cite{M.Mat1992}, \cite{SB1}, \cite{SB2}, \cite{SR} etc.) have studied various properties of $(\alpha,\beta)$-metrics.  The study of various types of curvatures of Finsler spaces such as $S$-curvature, mean Berwald curvature, flag curvature always remain the central idea in the Finsler geometry.   Z. Shen (\cite{S}) introduced the concept of  $S$-curvature of a Finsler space, in 1997. Cheng and  Shen (\cite{CS}) have given the formula of $S$-curvature of the Finsler space with $(\alpha, \beta)$-metrics, in 2009. \\ 
	
	  Finsler geometry has been developing rapidly since last few decades, after its emergence in 1917 (\cite{Finsler}). Finsler geometry has been influenced by group theory. The celebrated Erlangen program of F. Klein, posed in 1872 (\cite{Klein}), greatly influenced the development of geometry. Klein proposed to categorize the geometries by their chacteristic group of transformations. The Myers-Steenrod theorem, published in 1939 (\cite{Myers-Steenrod}), extended the scope of applying  Lie theory to all homogeneous Riemannian manifolds.\\
	
	\textbf{Theorem}({\bf Myers-Steenrod}): \textit{Let $ M $ be a connected Riemannian manifold. Then the group of isometries $ I(M) $ of $ M $ admits a differentiable structure such that $ I(M) $ is a Lie transformation group of $ M $}.\\
	
	 S. Deng and Z. Hou (\cite{DH})  have generalized this theorem to the Finslerian case, in 2002. This result opened the door for applying Lie theory to study Finsler geometry. The current topics of research  in Finsler geometry are homogeneous Finsler spaces, Finsler spaces with $(\alpha, \beta)$-metrics and symmetric spaces etc.\\
	 
	 To compute the geometric quantities, specially, curvatures is an interesting problem in homogeneous spaces. In 1976, Milnor (\cite{Milnor}) studied the curvature properties of such spaces by using the formula for the sectional curvature of a left invariant Riemannian metric on a Lie group.\\ 
	 
	 S. Deng (\cite{S-cur D}) has derived an explicit formula to find $S$-curvature of homogeneous Randers metric, in 2009. Later, Deng and Wang (\cite{S-cur DW}) deduced a formula for  $S$-curvature and mean Berwald curvature of the homogeneous Finsler space with an $(\alpha, \beta)$-metric, in 2010. \\
	   The paper consists of five  sections arranged as follows:\\
	   
	 Section 2 includes some preliminaries of Finsler geometry. In third section, we  prove the existence of  invariant vector field in a homogeneous Finsler space with infinite series $(\alpha, \beta)$-metric and also for exponential metric. In fourth section,  we derive a formula for $S$-curvature of homogeneous Finsler space with $(\alpha, \beta)$-metric, without using local coordinates and also deduce a formula for $S$-curvature of homogeneous Finsler space with infinite series $(\alpha, \beta)$-metric and for exponential metric. Further, using it, we prove that  homogeneous Finsler space with infinite series $(\alpha, \beta)$-metric has isotropic $S$-curavture if and only if it has vanishing $S$-curvature. Finally, in the last section, we  deduce the formula for mean Berwald curvature of homogeneous Finsler space with afore said metrics. \\

	\section{Preliminaries}
		In this section, we give some basic concepts of Finsler geometry that are required for next sections. For symbols and notations , we refer  (\cite{BCS}, \cite{CSBOOK})  and (\cite{Homogeneous Finsler Spaces}).
	\begin{definition}
		Let $M$ be a  smooth manifold of dimension $n, T_pM$  the tangent space at any point $p \in M.$ A real valued  bilinear function   
		$g\colon T_pM\times T_pM\longrightarrow [0,\infty)$
		is called  a Riemannian metric if it is symmetric and positive-definite,i.e., $\forall\;X,Y \in \mathfrak{X}(M), $ 
		\begin{enumerate}
			\item[\bf(i)]$ g(X,Y)=g(Y,X).$
			\item[\bf(ii)] $ g(X,X)\geq0$ and $ g(X,X)=0$ if and only if $X=0.$ 
		\end{enumerate}
		A smooth manifold with a given Riemannian metric is called a \textbf{Riemannian manifold.}
	\end{definition}
	\begin{definition}
		An n-dimensional real vector space $V$ is called a \textbf{Minkowski space}
		if there exists a real valued function $F:V \longrightarrow \mathbb{R}$ satisfying the following conditions: 
		\begin{enumerate}
			\item[\bf(a)]  $F$ is smooth on $V \backslash \{0\},$ 
			\item[\bf(b)] $F(v) \geq 0  \ \ \forall \ v \in V,$
			\item[\bf(c)] $F$ is positively homogeneous, i.e., $ F(\lambda v)= \lambda F(v), \ \ \forall \ \lambda > 0, $
			\item[\bf(d)] For a basis $\{v_1,\ v_2, \,..., \ v_n\}$ of $V$ and $y= y^iv_i \in V$, the Hessian matrix $\left( g_{_{ij}}\right)= \left( \dfrac{1}{2} F^2_{y^i y^j} \right)  $ is positive-definite at every point of $V \backslash \{0\}.$
		\end{enumerate} 
	Here,	$F$ is called a Minkowski norm.
	\end{definition}
	\begin{definition}	
		A connected smooth manifold $M$ is called a \textbf{Finsler space} if there exists a function $F\colon TM \longrightarrow [0, \infty)$ such that $F$ is smooth on the slit tangent bundle $TM \backslash \{0\}$ and the restriction of $F$ to any $T_p(M), \ p \in M$, is a Minkowski norm. In this case, $F$ is called a Finsler metric. 
	\end{definition}
	Let $(M, F)$ be a Finsler space and let $(x^i,y^i)$ be a standard coordinate system in $T_x(M)$. The induced inner product $g_y$ on $T_x(M)$ is given by $g_y(u,v)=g_{ij}(x,y)u^i v^j$, where $u=u^i \dfrac{\partial}{\partial x^i}, \ v=v^i\dfrac{\partial}{\partial x^i} \in T_x(M) $. Also note that $F(x,y)= \sqrt{g_y(y,y)}.$  \\
	The  condition for an $(\alpha,\beta)$-metric to be a Finsler metric is given in following Shen's lemma:
	\begin{lemma}(\cite{CSBOOK}) {\label{existence of metric}}
		Let $F=\alpha \phi(s), \ s=\beta/ \alpha,$ where $\alpha$ is a Riemannian metric and $\beta$ is a 1-form whose length with respect to $\alpha$ is bounded above, i.e., $b:=\lVert \beta \rVert_{\alpha} < b_0,$ where $b_0$ is a positive real number. Then $F$ is a Finsler metric if and only if the function $\phi=\phi(s)$ is a smooth positive function on $\left( -b_0, b_0\right) $ and satisfies the following condition:
		$$ \phi(s)-s\phi'(s)+\left( b^2-s^2\right) \phi''(s)>0, \ \ \lvert s\rvert \leq b < b_0.$$
		
	\end{lemma}
	
\begin{definition}
	Let $ (M, F) $ be a Finsler space. A diffeomorphism $\phi\colon  M \longrightarrow M$ is called an \textbf{isometry} if $F\left( \phi(p),  d\phi_p(X)\right) =F(p,X)$ for any $p \in M$ and $X \in T_p(M).$
\end{definition}	
	
	Let $\{e_i\}$ be a basis of an n-dimensional real vector space $V$ and $F$ be a Minkowski norm on $V$. We denote the volume of a subset in $\mathbb{R}^n$ by Vol and $B^n$, the open unit ball. The quantity $$\tau(y)=\ln \dfrac{\sqrt{det(g_{ij}(y))}}{\sigma_F}, \ \ y\in V\backslash \{0\}, $$
	where
	$$ \sigma_F= \dfrac{Vol\left( B^n\right) }{Vol\left\lbrace \left( y^i \right)\in \mathbb{R}^n : F\left( y^ie_i\right) < 1   \right\rbrace },  $$
	is called the distortion of $(V, F).$ Further, let $\tau(x,y)$ be the distortion of $F$ on $T_x(M).$ For any tangent vector $y \in T_x(M)\backslash \{0\}, $ let $\gamma(t)$ be the geodesic such that $\gamma(0)=x, \ \dot{\gamma}(0)=y.$ The rate of change of distortion along the geodesic $\gamma$ is called the \textbf{$S$-curvature} and it is denoted by $S(x,y)$, i.e., 
	$$ S(x,y)= \dfrac{d}{dt}\bigg[\tau\bigg( \gamma(t), \dot{\gamma}(t)\bigg) \bigg]_{t=0}. $$
	This quantity is positively homogeneous of degree one, i.e., $S(x,\lambda y)=\lambda S(x,y), \ \  \lambda>0.$
	We can observe that any Riemannian manifold has vanishing $S$-curvature. Therefore, we can say that $S$-curvature is a non-Riemannian quantity. Further, there is another quantity that is related to $S$-curvature, called $E$-curvature or mean Berwald curvature. The mean Berwald curvature is given by
	$$ E_{ij}(x,y)= \dfrac{1}{2}\dfrac{\partial^2S(x,y)}{\partial y^i \partial y^j}.$$
\begin{definition}
	Let $ G $ be a Lie group and $ M $ a smooth manifold. If $ G $ has a smooth action on $ M $, then $ G $ is called a Lie transformation group of $ M $.
\end{definition}
	\begin{definition}
		A connected Finsler space $ (M, F)$  is said to be homogeneous Finsler space if the action of the group of isometries of $(M,F)$, denoted by $ I(M, F)$, is transitive on $M$. 
	\end{definition}

	\section{Invariant Vector Field}
	In this section, first we prove the existence of invariant vector field corresponding to 1-form $\beta$ for a homogeneous Finsler space with infinite series  $(\alpha, \beta)$-metric $F= \dfrac{\beta^2}{\beta-\alpha}$. For this, first we prove the following lemma:
	
	\begin{lemma}{\label{a1}}
		
	Let $ (M,\alpha)$ be a Riemannian space. Then the infinite series Finsler metric $F= \dfrac{\beta^2}{\beta-\alpha}$, $\beta=b_i y^i$, a  1-form with $\lVert \beta \rVert = \sqrt{b_i b^i} < 1$ consists of a Riemannian metric $\alpha$ alongwith a smooth vector field $X$ on $M$ with $\alpha\left( X\rvert_x\right) < 1$, $ \forall \  x \in M$, i.e.,  $$ F\left( x, y\right)=  \dfrac{\left\langle  X\rvert_x, y\right\rangle^2}{\left\langle  X\rvert_x, y\right\rangle - \alpha\left( x, y\right) }, \ \ x \in M, \ \ y \in T_x M,$$ 
	where $\left\langle \ , \ \right\rangle $ is the inner product induced by the Riemannian metric $\alpha.$
	\end{lemma}
	\begin{proof}
     The definition of Riemannian metric assures that the bilinear form $ \left\langle u, v \right\rangle = a_{ij}u^i v^j, \ \ u, v \ \in  T_x M  $ is an inner product on $T_x M$. Further, this inner product induces an inner product on the cotangent space $ T^*_x (M).$ With the help of this inner product, we can define a linear isomorphism between  $T^*_x (M)$ and $ T_x M$.
     Thus the 1-form $\beta$ corresponds to a smooth vector field $X$ on $M$, given by $$ X\rvert_x = b^i \dfrac{\partial}{\partial x_i}, \text{where} \ \   b^i= a^{ij} b_j.$$ 
     Finally, we have
     $$ \left\langle  X\rvert_x, y\right\rangle = \left\langle b^i \dfrac{\partial}{\partial x_i} \ , \  y^j \dfrac{\partial}{\partial x_j} \right\rangle = b^i y^j a_{ij} = b_j y^j = \beta(y).$$
     Also, note that $$\alpha\left( X\rvert_x\right) = \lVert \beta \rVert < 1.$$
	\end{proof}
	
	\begin{lemma}{\label{a2}}
		Let $ (M,F)$ be a Finsler space with infinite series Finsler metric  $F= \dfrac{\beta^2}{\beta-\alpha} $. Then the group of isometries $ I(M, F)$ of $ (M,F)$ is a closed subgroup of the group of isometries $ I(M, \alpha)$ of the  Riemannian  space $ (M,\alpha).$   	
	\end{lemma}
	\begin{proof}
	Let $\psi$ be an isometry of $(M,F)$ and let $ x \in M.$ Therefore,  for every $ y \in T_x M$, we have 
$$ F(x,y) = F( \psi(x), d \psi_x(y)).$$
Applying  Lemma \ref{a1}, we get
\begin{equation*} 
  \dfrac{\left\langle  X\rvert_x, y\right\rangle^2}{\left\langle  X\rvert_x, y\right\rangle - \alpha\left( x, y\right) } =\dfrac{\left\langle  X\rvert_{\psi(x)}, d \psi_x(y)\right\rangle^2}{\left\langle  X\rvert_{\psi(x)}, d \psi_x(y)\right\rangle - \alpha\left( \psi(x), d \psi_x(y)\right) },
  \end{equation*}
  which implies
  \begin{equation}{\label{e1}}
  \begin{split}
 \left\langle  X\rvert_x, y\right\rangle^2 \left\langle  X\rvert_{\psi(x)}, d \psi_x(y)\right\rangle - \left\langle  X\rvert_x, y\right\rangle^2 \alpha\left( \psi(x), d \psi_x(y)\right)  \\
 = \left\langle  X\rvert_x, y\right\rangle \left\langle  X\rvert_{\psi(x)}, d \psi_x(y)\right\rangle^2 - \alpha\left( x, y\right) \left\langle  X\rvert_{\psi(x)}, d \psi_x(y)\right\rangle^2.
 \end{split} 
 \end{equation}
 Replacing $y $ by $-y$ in equation (\ref{e1}), we get
 \begin{equation}{\label{e2}}
 \begin{split}
 \left\langle  X\rvert_x, y\right\rangle^2 \left\langle  X\rvert_{\psi(x)}, d \psi_x(y)\right\rangle + \left\langle  X\rvert_x, y\right\rangle^2 \alpha\left( \psi(x), d \psi_x(y)\right)  \\
 = \left\langle  X\rvert_x, y\right\rangle \left\langle  X\rvert_{\psi(x)}, d \psi_x(y)\right\rangle^2 + \alpha\left( x, y\right) \left\langle  X\rvert_{\psi(x)}, d \psi_x(y)\right\rangle^2. 
 \end{split}
 \end{equation}
 Adding  equations (\ref{e1}) and (\ref{e2}), we get 
 \begin{equation}{\label{e3}}
  \left\langle  X\rvert_x, y\right\rangle = \left\langle  X\rvert_{\psi(x)}, d \psi_x(y)\right\rangle. 
  \end{equation}
 Subtracting  equation (\ref{e2}) from equation (\ref{e1}) and using  equation (\ref{e3}), we get 
 \begin{equation}
\alpha\left( x, y\right)= \alpha\left( \psi(x), d \psi_x(y)\right).
 \end{equation}
 Therefore $ \psi $ is an isometry with respect to the Riemannian metric $\alpha$ and $ d \psi_x\left( X\rvert_x\right) = X\rvert_{\psi(x)}.$ Thus $ I(M,F)$ is a closed subgroup of $ I(M,\alpha).$ 
	\end{proof}

     Let $ (M,F)$ be a  homogeneous Finsler space with infinite series Finsler metric  $F= \dfrac{\beta^2}{\beta-\alpha} $, then from the Lemma (\ref{a2}), it is clear that the Riemannian manifold $(M, \alpha)$ is homogeneous. $M$ can be written as a coset space $G/H$, where $G= I(M,F)$ is a Lie transformation group of $M$ and $H$,  the compact isotropy subgroup of $ I(M,F)$ at some point $x \in M$(\cite{DH}) . Let $ \mathfrak g$ and $ \mathfrak h $ be the Lie algebras of the Lie groups $G$ and $H$ respectively. If $ \mathfrak g$ can be written as a direct sum of subspace $\mathfrak h$ and subspace $\mathfrak m$ of $\mathfrak g$   such that $\text{Ad}(h) \mathfrak m \subset \mathfrak m \ \ \forall \  h \in H,$ then $(G/H,F)$ is called a reductive homogeneous manifold (\cite{N}). Note that a  Finsler metric $F$ can be viewed as a $G$-invariant Finsler metric on $M$. Thus, we can say that any homogeneous Finsler manifold can be written as a coset space of a connected Lie group with an invariant Finsler metric.\\
     Next, we prove the following lemma, necessary for further calculations.
     \begin{lemma}
     Let   $F= \dfrac{\beta^2}{\beta-\alpha}$ be a $G$-invariant infinite series metric on $G/H$. then $ \alpha  $ is a  $G$-invariant Riemannian metric and the vector field $X$ corresponding to the 1-form $\beta $ is a $G$-invariant vector field. 	
     \end{lemma}
     \begin{proof}
   Since $F$ is $G$-invariant, we have
   $$ F\left( y\right)=F\left( \text{Ad} \left(h \right) y \right), \; \; \forall \;  h \in H, \; \;  y \in \mathfrak m.  $$ 
   Applying the Lemma \ref{a1}, we get
   \begin{equation*} 
   \dfrac{\left\langle  X, y\right\rangle^2}{\left\langle  X, y\right\rangle - \alpha\left( y\right) } =\dfrac{\left\langle  X, \text{Ad}\left(h \right) y\right\rangle^2}{\left\langle  X, \text{Ad}\left(h \right) y\right\rangle - \alpha\left( \text{Ad}\left(h \right) y\right) }.
   \end{equation*}
   Simplifying the above equation, we get
   \begin{equation}{\label{e5}}
   \begin{split}
   \left\langle  X, y\right\rangle^2 \left\langle  X, \text{Ad}\left(h \right) y\right\rangle - \left\langle  X, y\right\rangle^2\alpha\left( \text{Ad}\left(h \right) y\right) \\ \; \; \; \; \; \; \; \; = \left\langle  X, \text{Ad}\left(h \right) y\right\rangle^2 \left\langle  X, y\right\rangle - \left\langle  X, \text{Ad}\left(h \right) y\right\rangle^2\alpha\left( y\right).
   \end{split} 
   \end{equation}
   Replacing $y $ by $-y$ in  equation (\ref{e5}), we get
   \begin{equation}{\label{e6}}
   \begin{split}
   \left\langle  X, y\right\rangle^2 \left\langle  X, \text{Ad}\left(h \right) y\right\rangle + \left\langle  X, y\right\rangle^2\alpha\left( \text{Ad}\left(h \right) y\right) \\ \; \; \; \; \; \; \; \; = \left\langle  X, \text{Ad}\left(h \right) y\right\rangle^2 \left\langle  X, y\right\rangle + \left\langle  X, \text{Ad}\left(h \right) y\right\rangle^2\alpha\left( y\right).
   \end{split} 
   \end{equation}
   Adding  equations (\ref{e5}) and (\ref{e6}), we get 
   \begin{equation}{\label{e7}}
   \left\langle  X, \text{Ad}\left(h \right) y\right\rangle = \left\langle  X, y\right\rangle. 
   \end{equation}
   Subtracting  equation (\ref{e6}) from equation (\ref{e5}) and using  equation (\ref{e7}), we get 
   \begin{equation}
   \alpha\left(  y\right)= \alpha\left( \text{Ad}\left(h \right) y\right).
   \end{equation}
   Therefore $ \alpha $ is a $G$-invariant Riemannian metric and $ \text{Ad}\left(h \right) X= X.$ 
     \end{proof}
 
 Next, we prove the similar results for a homogeneous Finsler space with exponential metric.
 Let $ F= \alpha \phi(s),$ where $ \phi(s)=e^s,\  s=\beta / \alpha,\ \alpha = \sqrt{a_{ij}(x)y^i y^j}$ is a Riemannian metric and $ \beta = b_i(x) y^i,$ a differential 1-form. One can easily verify that $ F= \alpha e^{\beta/ \alpha}$ satisfies Lemma \ref{existence of metric}. So, it is a Finsler $(\alpha, \beta)$-metric. In literature, this metric is known as exponential metric.
 
 \begin{lemma}{\label{exp1}}
 	
 	Let $ (M,\alpha)$ be a Riemannian space. Then the exponential Finsler metric $F= \alpha  e^{\beta/\alpha}$, $\beta=b_i y^i$, a  1-form with $\lVert \beta \rVert = \sqrt{b_i b^i} < 1$ consists of a Riemannian metric $\alpha$ alongwith a smooth vector field $X$ on $M$ with $\alpha\left( X\rvert_x\right) < 1$, $ \forall \  x \in M$, i.e.,  $$ F\left( x, y\right)= \alpha(x,y)\ e^{\left\langle  X\rvert_x, y\right\rangle/\alpha(x,y)} , \ \ x \in M, \ \ y \in T_x M,$$ 
 	where $\left\langle \ , \ \right\rangle $ is the inner product induced by the Riemannian metric $\alpha.$
 \end{lemma}
 \begin{proof}
 	The definition of Riemannian metric assures that the bilinear form $ \left\langle u, v \right\rangle = a_{ij}u^i v^j, \ \ u, v \ \in  T_x M  $ is an inner product on $T_x M$. Further, this inner product induces an inner product on the cotangent space $ T^*_x (M).$ With the help of this inner product, we can define a linear isomorphism between  $T^*_x (M)$ and $ T_x M$.
 	Thus the 1-form $\beta$ corresponds to a smooth vector field $X$ on $M$, given by $$ X\rvert_x = b^i \dfrac{\partial}{\partial x_i}, \text{where} \ \   b^i= a^{ij} b_j.$$ 
 	Finally, we have
 	$$ \left\langle  X\rvert_x, y\right\rangle = \left\langle b^i \dfrac{\partial}{\partial x_i} \ , \  y^j \dfrac{\partial}{\partial x_j} \right\rangle = b^i y^j a_{ij} = b_j y^j = \beta(y).$$
 	Also, note that $$\alpha\left( X\rvert_x\right) = \lVert \beta \rVert < 1.$$
 \end{proof}
 	\begin{lemma}{\label{exp2}}
 	Let $ (M,F)$ be a Finsler space with  exponential  metric  $F=\alpha e^{\beta/\alpha}  $. Then the group of isometries $ I(M, F)$ of $ (M,F)$ is a closed subgroup of the group of isometries $ I(M, \alpha)$ of the  Riemannian  space $ (M,\alpha).$   	
 \end{lemma}
 \begin{proof}
 	Let $\psi$ be an isometry of $(M,F)$ and let $ x \in M.$ Therefore,  for every $ y \in T_x M$, we have 
 	$$ F(x,y) = F( \psi(x), d \psi_x(y)).$$
 	Applying Lemma \ref{exp1}, we get
 	\begin{equation} {\label{eexp1}}
 	\alpha(x,y)\ e^{\left\langle  X\rvert_x, y\right\rangle/\alpha(x,y)}= \alpha(\psi(x),d\psi_x(y)) \ e^{\left\langle  X\rvert_{\psi(x)}, d \psi_x(y)\right\rangle/\alpha(\psi(x),d\psi_x(y))}.
 	\end{equation}
 	Replacing $y $ by $-y$ in equation (\ref{eexp1}), we get
 	\begin{equation}{\label{eexp2}}
 	\alpha(x,y)\ e^{-\left\langle  X\rvert_x, y\right\rangle/\alpha(x,y)}= \alpha(\psi(x),d\psi_x(y)) \ e^{-\left\langle  X\rvert_{\psi(x)}, d \psi_x(y)\right\rangle/\alpha(\psi(x),d\psi_x(y))}.
 	\end{equation}
 	From   equations (\ref{eexp1}) and (\ref{eexp2}), we get
 	\begin{equation} {\label{eexp3}}
 	e^{2\left\langle  X\rvert_x, y\right\rangle/\alpha(x,y)}=  e^{2\left\langle  X\rvert_{\psi(x)}, d \psi_x(y)\right\rangle/\alpha(\psi(x),d\psi_x(y))},
 	\end{equation}
 	which further implies 
 	\begin{equation} {\label{eexp4}}
 	\dfrac{\left\langle  X\rvert_x, y\right\rangle}{\alpha(x,y)}=  \dfrac{\left\langle  X\rvert_{\psi(x)}, d \psi_x(y)\right\rangle}{\alpha(\psi(x),d\psi_x(y))}.
 	\end{equation}
 	From  equations (\ref{eexp1}) and (\ref{eexp4}), we have
 	\begin{equation} {\label{eexp5}}
 	\alpha(x,y)= \alpha(\psi(x),d\psi_x(y)).
 	\end{equation}
 	Further,	from  equations (\ref{eexp4}) and (\ref{eexp5}), we have
 	$$\left\langle  X\rvert_x, y\right\rangle=\left\langle  X\rvert_{\psi(x)}, d \psi_x(y)\right\rangle. $$
 	Therefore $ \psi $ is an isometry with respect to the Riemannian metric $\alpha$ and $ d \psi_x\left( X\rvert_x\right) = X\rvert_{\psi(x)}.$ Thus $ I(M,F)$ is a closed subgroup of $ I(M,\alpha).$ 
 \end{proof}
  Let $ (M,F)$ be a  homogeneous Finsler space with exponential Finsler metric  $F= \alpha   e^{\beta/\alpha} $, Then from the Lemma (\ref{exp2}), it is clear that the Riemannian manifold $(M, \alpha)$ is homogeneous. Further, $M$ can be written as a coset space $G/H$, where $G= I(M,F)$ is a Lie transformation group of $M$ and $H$,  the compact isotropy subgroup of $ I(M,F)$ at some point $x \in M$(\cite{DH}). Let $ \mathfrak g$ and $ \mathfrak h $ be the Lie algebras of the Lie groups $G$ and $H$ respectively. If $ \mathfrak g$ can be written as a direct sum of subspace $\mathfrak h$ and subspace $\mathfrak m$ of $\mathfrak g$   such that $\text{Ad}(h) \mathfrak m \subset \mathfrak m \ \ \forall \ h \in H,$ then  $(G/H,F)$ is called a reductive homogeneous manifold (\cite{N}). Note that a  Finsler metric $F$ can be viewed as a $G$-invariant Finsler metric on $M$. Thus, we can say that any homogeneous Finsler manifold can be written as a coset space of a connected Lie group with an invariant Finsler metric.\\
  
 Next, we first prove the following theorem, and then establish the existence of one-to-one correspondence between the invariant vector fields on $G/H$ and its subspace.
 \begin{theorem}
 	Let   $F= \alpha  e^{\beta/\alpha}$ be a $G$-invariant exponential metric on $G/H$. Then $ \alpha  $ is a  $G$-invariant Riemannian metric and the vector field $X$ corresponding to the 1-form $\beta $ is a $G$-invariant vector field. 	
 \end{theorem}
 \begin{proof}
 	Since $F$ is $G$-invariant, we have
 	$$ F\left( y\right)=F\left( \text{Ad} \left(h \right) y \right), \; \; \forall \;  h \in H, \; \;  y \in \mathfrak m.  $$ 
 	Applying the Lemma \ref{exp1}, we get
 	\begin{equation} {\label{eexp13}}
 	\alpha(y)\ e^{\left\langle  X, y\right\rangle/\alpha(y)}= \alpha(\text{Ad} \left(h \right)y) \ e^{\left\langle  X, \text{Ad} \left(h \right)y\right\rangle/\alpha(\text{Ad} \left(h \right)y)}.
 	\end{equation}
 	Replacing $y $ by $-y$ in equation (\ref{eexp13}), we get
 	\begin{equation}{\label{eexp14}}
 	\alpha(y)\ e^{-\left\langle X, y\right\rangle/\alpha(y)}= \alpha(\text{Ad} \left(h \right)y) \ e^{-\left\langle  X, \text{Ad} \left(h \right)y\right\rangle/\alpha(\text{Ad} \left(h \right)y)}.
 	\end{equation}
 	From  equations (\ref{eexp13}) and (\ref{eexp14}), we get
 	\begin{equation} {\label{eexp15}}
 	e^{2\left\langle  X, y\right\rangle/\alpha(y)}=  e^{2\left\langle  X, \text{Ad} \left(h \right)y\right\rangle/\alpha(\text{Ad} \left(h \right)y)},
 	\end{equation}
 	which further implies 
 	\begin{equation} {\label{eexp16}}
 	\dfrac{\left\langle  X, y\right\rangle}{\alpha(y)}=  \dfrac{\left\langle  X, \text{Ad} \left(h \right)y\right\rangle}{\alpha(\text{Ad} \left(h \right)y)}.
 	\end{equation}
 	From  equations (\ref{eexp13}) and (\ref{eexp16}), we have
 	\begin{equation} {\label{eexp17}}
 	\alpha(y)= \alpha(\text{Ad} \left(h \right)y).
 	\end{equation}
 	Further, from  equations (\ref{eexp16}) and (\ref{eexp17}), we have
 	$$\left\langle  X, y\right\rangle=\left\langle  X, \text{Ad} \left(h \right)y\right\rangle. $$
 	Therefore $ \alpha $ is a $G$-invariant Riemannian metric and $ \text{Ad}\left(h \right) X= X.$ 
 \end{proof}
 
\begin{definition}
	A one-parameter subgroup of a Lie group $G$ is a smooth map $\phi : \mathbb{R} \longrightarrow G$ such that $\phi(0)=e$ and $\phi\left(t_1+t_2 \right)= \phi\left(t_1 \right)\phi\left(t_2 \right) \ \ \forall \ t_1,\  t_2 \in \mathbb{R},  $ where $e$ is the identity of $G$.
\end{definition}
The following result,  proved in (\cite{Homogeneous Finsler Spaces})  guarantees the existence of one-parameter subgroup of a Lie group.
\begin{theorem}
	Let $G$ be a Lie group with $\mathfrak g$ as its Lie algebra. Then for any $v \in \mathfrak g$,  there exists a unique one-parameter subgroup $\phi_v$ such that $\dot{\phi}_v (0)=v_e$, where $e$ is the identity of $G$.
\end{theorem} 
	\begin{definition}
	The exponential map $\exp: \mathfrak g \longrightarrow G$ is defined by $$\exp(tv)= \phi_v(t), \ \ \forall \ t \in \mathbb R.$$
	\end{definition}
We can identify the tangent space $T_{eH}\left( G/H\right) $ of $G/H$ at the origin $eH=H$ with $\mathfrak m$ through the following map:
$$
\mathfrak m \longrightarrow T_{eH}\left( G/H\right) $$
$$ \ \ \ \ \ \ \ \ \ \ v \longrightarrow \dfrac{d}{dt}\left( \exp(tv)H\right) \rvert_{t=0}. $$
Observe that for any $v \in \mathfrak g$, the vector field $ \tilde{v}= \dfrac{d}{dt}\left( \exp(tv)H\right) \rvert  _{t=0}  $ is called the fundamental Killing vector field generated by $v$ (\cite{KN}). 
The following proposition, proved in (\cite{Invariant}),  gives a complete description of invariant vector fields.
\begin{proposition}
	There exists a one to one correspondence between the set of invariant vector fields on $G/H $ and the subspace $$ V= \left\lbrace v  \in \mathfrak m : \text{Ad} \left( h\right) v= v, \; \forall \; h  \in H \right\rbrace. $$
	\end{proposition}
\section{$S$-Curvature}
The notion of $S$-curvature of a Finsler space is closely  associated with a volume form. There are mainly two important volume forms  in Finsler geometry, namely : the Busemann-Hausdorff volume form and the Holmes-Thompson volume form.\\
The  Busemann-Hausdorff volume form $ dV_{BH}=\sigma_{_{BH}}(x)dx$ is given by
$$ \sigma_{_{BH}}(x)= \dfrac{Vol\left( B^n\right) }{Vol\left\lbrace \left( y^i\right) \in \mathbb{R}^n \ \colon\  F\left( x, y^i\frac{\partial}{\partial x^i}\right) < 1 \right\rbrace }.$$
The Holmes-Thompson volume form $ dV_{HT}=\sigma_{_{HT}}(x)dx$ is given by
$$ \sigma_{_{HT}}(x)= \dfrac{1}{Vol\left( B^n\right)} \int_{\left\lbrace \left( y^i\right) \in \mathbb{R}^n \ \colon\  F\left( x, y^i\frac{\partial}{\partial x^i}\right) < 1 \right\rbrace }det\left( g_{ij}\right) dy  .$$
Particularly, if $F= \sqrt{g_{ij}(x)y^i y^j}$ is a Riemannian metric, then both the volume forms are equal to the Riemannian volume form, i.e,  $ dV_{HT}=dV_{BH}=\sqrt{det\left( g_{ij}(x) \right) } dx$.\\ 
Next, consider a function $$ T(s)= \phi \left( \phi-s \phi'\right)^{n-2}\left\lbrace  \left( \phi-s \phi'\right) + \left( b^2-s^2\right) \phi'' \right\rbrace . $$
and let $dV_{\alpha}= \sqrt{det\left( a_{ij} \right) } dx $  be the Riemannian volume form of $\alpha.$ Then the volume form $dV= dV_{BH}$ or $dV_{HT}$ is given by $dV= f(b) dV_{\alpha}$, where 
$$
f(b)=  \begin{cases}
\dfrac{\int_{0}^{\pi} \sin^{n-2}t \ dt}{\int_{0}^{\pi}\dfrac{\sin^{n-2} t}{\phi\left( b \cos t\right)^n } \ dt}, & \text{if}\ \  dV= dV_{BH}, \\ \\
\dfrac{ \int_{0}^{\pi}\left( \sin^{n-2}\right) T\left( b \cos t\right) \ dt }{\int_{0}^{\pi} \sin^{n-2} t \ dt}, & \text{if}\ \  dV= dV_{HT}.
\end{cases} 
$$
Cheng and Shen (\cite{CS}) have derived the formula for the $S$-curvature of  $(\alpha, \beta)$-metric in a local coordinate system, given by
\begin{equation}{\label{eqS}}
S= \left( 2\psi- \dfrac{f'(b)}{bf(b)}\right)\left(r_0+s_o \right) - \dfrac{\Phi}{2 \alpha \Delta^2}\bigg( r_{_{00}}-2\alpha  Qs_{_{0}}\bigg),   
\end{equation}
where
\begin{align*}
Q&= \dfrac{\phi'}{\phi-s \phi'} \ ,\\
\Delta &= 1+sQ+\left(b^2-s^2\right)Q'\ ,\\
\Phi&= -\left( Q-sQ'\right) \left( n\Delta + 1 + sQ\right)- \left(b^2-s^2 \right) \left(1+sQ \right)Q'',\\
\psi&= \dfrac{Q'}{2 \Delta},\\
r_{ij}&=\dfrac{1}{2} \left(b_{i\mid j}+b_{j \mid i} \right), \\
s_{ij}&= \dfrac{1}{2} \left(b_{i\mid j}-b_{j \mid i} \right), \\
r_j &= b^ir_{ij}\ , \ \ s_j= b^i s_{ij} , \\
r_0 &=r_i y^i \ , \ \ s_0= s_i y^i\ , \ \ , r_{00}= r_{ij} y^i y^j.
\end{align*} 
Also, in (\cite{CS}), it is proved that if the Riemannian length $b$ is constant, then $r_0+s_0=0$. Hence, in this case,  $S$-curvature of  an $(\alpha, \beta)$-metric $F= \alpha \phi(s), \ s=\dfrac{\beta}{\alpha}$ is given by  
\begin{equation}{\label{e9}}
S= - \dfrac{\Phi}{2 \alpha \Delta^2}\bigg(r_{_{00}}-2 \alpha Qs_{_{0}} \bigg). 
\end{equation} 

In case of a homogeneous Finsler space, $b$ is constant. Therefore, the $S$-curvature of a homogeneous Finsler space with $(\alpha, \beta)$-metric can be expressed by the equation  (\ref{e9}).
Since the Finsler space is homogeneous, it is sufficient to  compute $S$-curvature of a homogeneous Finsler space with$(\alpha,\beta)$-metric at origin $H$ as follows:\\

Let $v$ be a $G$-invariant vector field in $\mathfrak m$ corresponding to 1-form $\beta$ with length $c= \lvert v \rvert$. Also, let $\{v_1, v_2,..., v_n\}$ be an orthonormal basis of $\mathfrak m$ such that $v_n= \dfrac{v}{c}$. Then there exists a neighbourhood $N$ of origin $eH=H$ in $G/H$ such that the map 
$$ \left( \exp x^1 v_1 \exp x^2 v_2 ... \exp x^n v_n\right)H \longmapsto \left(x^1,  x^2, ... ,  x^n\right)   $$ 
defines a local coordinate system on $N$(\cite{H}).
In \cite{S-cur D}, it is proved that $\tilde v= c \dfrac{\partial}{\partial x^n}.$ Now, we calculate $b_i.$
$$ b_i= \beta \left( \dfrac{\partial}{\partial x^i}\right)=\left\langle \tilde v, \dfrac{\partial}{\partial x^i} \right\rangle = c \left\langle \dfrac{\partial}{\partial x^n}, \dfrac{\partial}{\partial x^i} \right\rangle.  $$
Clearly, $b_n=c$ and all other $b_i=0 \ \forall \  i\neq n$ at the origin.\\
Further, 
\begin{align*}
\dfrac{\partial b_i}{\partial x^j} &= c \dfrac{\partial}{\partial x^j} \left\langle \dfrac{\partial}{\partial x^n}, \dfrac{\partial}{\partial x^i} \right\rangle\\
&= c\left( \left\langle \nabla_{\dfrac{\partial}{\partial x^j}} \dfrac{\partial}{\partial x^n}, \dfrac{\partial}{\partial x^i} \right\rangle+ \left\langle \dfrac{\partial}{\partial x^n},\nabla_{\dfrac{\partial}{\partial x^j}} \dfrac{\partial}{\partial x^i} \right\rangle \right), 
\end{align*}
and
\begin{align*}
\dfrac{\partial b_j}{\partial x^i} &= c \dfrac{\partial}{\partial x^i} \left\langle \dfrac{\partial}{\partial x^n}, \dfrac{\partial}{\partial x^j} \right\rangle\\
&= c\left( \left\langle \nabla_{\dfrac{\partial}{\partial x^i}} \dfrac{\partial}{\partial x^n}, \dfrac{\partial}{\partial x^j} \right\rangle+ \left\langle \dfrac{\partial}{\partial x^n},\nabla_{\dfrac{\partial}{\partial x^i}} \dfrac{\partial}{\partial x^j} \right\rangle \right). 
\end{align*}
The following formula proved in (\cite{S-cur D}), is required for further computations:
$$\Gamma^{l}_{ij}(H) = \dfrac{1}{2}\bigg(-\left\langle  \left[ v_i, v_j\right]_{\mathfrak m}, v_l \right\rangle  + \left\langle \left[ v_l, v_i\right]_{\mathfrak m}, v_j \right\rangle + \left\langle \left[ v_l, v_j\right]_{\mathfrak m}, v_i \right\rangle\bigg), \ \ i\geq j. $$
Next,
\begin{align*}
s_{ij}\left( H\right)&= \dfrac{1}{2} \left(b_{i\mid j}-b_{j \mid i} \right) \\
&=  \dfrac{1}{2} \left( \dfrac{\partial b_i}{\partial x^j}- b_k \Gamma^{k}_{ij}-\dfrac{\partial b_j}{\partial x^i}+ b_k \Gamma^{k}_{ji}\right) \\
&= \dfrac{c}{2}\left( \left\langle \nabla_{\dfrac{\partial}{\partial x^j}} \dfrac{\partial}{\partial x^n}, \dfrac{\partial}{\partial x^i} \right\rangle+ \left\langle \dfrac{\partial}{\partial x^n},\nabla_{\dfrac{\partial}{\partial x^j}} \dfrac{\partial}{\partial x^i} \right\rangle -\left\langle \nabla_{\dfrac{\partial}{\partial x^i}} \dfrac{\partial}{\partial x^n}, \dfrac{\partial}{\partial x^j} \right\rangle- \left\langle \dfrac{\partial}{\partial x^n},\nabla_{\dfrac{\partial}{\partial x^i}} \dfrac{\partial}{\partial x^j} \right\rangle\right) \\
&= \dfrac{c}{2}\left(  \left\langle \nabla_{\dfrac{\partial}{\partial x^n}} \dfrac{\partial}{\partial x^j}, \dfrac{\partial}{\partial x^i} \right\rangle - \left\langle \nabla_{\dfrac{\partial}{\partial x^n}} \dfrac{\partial}{\partial x^i}, \dfrac{\partial}{\partial x^j} \right\rangle + \left\langle \dfrac{\partial}{\partial x^n}, \left[ \dfrac{\partial}{\partial x^j}, \dfrac{\partial}{\partial x^i}\right] \right\rangle \right) \\
&= \dfrac{c}{2}\left( \left\langle \Gamma ^{k} _{nj} \ \dfrac{\partial}{\partial x^k}, \dfrac{\partial}{\partial x^i}\right\rangle - \left\langle \Gamma ^{k} _{ni} \  \dfrac{\partial}{\partial x^k}, \dfrac{\partial}{\partial x^j}\right\rangle \right) 
\end{align*}
\begin{align*}
&= \dfrac{c}{2}\left( \Gamma^{i}_{nj}- \Gamma^{j}_{ni}\right) \\
&= \dfrac{c}{4}\biggl\{ \bigg( -\left\langle \left[ v_n, v_j\right]_{\mathfrak m}, v_i \right\rangle + \left\langle \left[ v_i, v_n\right]_{\mathfrak m}, v_j \right\rangle + \left\langle \left[ v_i, v_j\right]_{\mathfrak m}, v_n \right\rangle  \bigg) \\
& \; \; \; \; \; \; \; \ \ \ - \bigg( -\left\langle  \left[ v_n, v_i\right]_{\mathfrak m}, v_j \right\rangle  + \left\langle \left[ v_j, v_n\right]_{\mathfrak m}, v_i \right\rangle + \left\langle \left[ v_j, v_i\right]_{\mathfrak m}, v_n \right\rangle  \bigg) \biggr\}\\
&= \dfrac{c}{2}\bigg<\left[ v_i, v_j\right] _{\mathfrak m}, v_n \bigg>.
\end{align*}
Further,
$$
s^i_{\ j}\left( H\right) = a^{ik}\left( H\right)s_{kj}\left( H\right)
= \sum_{k=1}^{n} \delta_{ik} s_{kj}\left( H\right) 
= s_{ij}\left( H\right)  $$
and 
$$ s_i\left( H\right)= b_l\left( H\right) s^l_{\ i}\left( H\right) = cs^n_{\ i}\left( H\right)= cs_{ni}\left( H\right). $$
Therefore, for $y= y^i v_{i} \in \mathfrak m$, we have
\begin{align*}
s_{_{0}}\left( y\right) &= s_i\left( H\right) y^i\\
&=c s_{ni}\left( H\right) y^i\\
&=\dfrac{c^2}{2} y^i \bigg<\left[ v_n, v_i\right] _{\mathfrak m}, v_n \bigg>\\
&= \dfrac{1}{2}  \bigg<\left[c v_n, y^i v_i\right] _{\mathfrak m}, c v_n \bigg>\\
&= \dfrac{1}{2}  \bigg<\left[v, y\right] _{\mathfrak m}, v \bigg>.
\end{align*}
Further,
\begin{align*}
r_{ij}\left( H\right) &=\dfrac{1}{2} \left(b_{i\mid j}+b_{j \mid i} \right) \\
&=  \dfrac{1}{2} \left( \dfrac{\partial b_i}{\partial x^j}- b_k \Gamma^{k}_{ij}+\dfrac{\partial b_j}{\partial x^i}- b_k \Gamma^{k}_{ji}\right) \\ 
&= \dfrac{1}{2} \left( \dfrac{\partial b_i}{\partial x^j}+\dfrac{\partial b_j}{\partial x^i}- 2b_n \Gamma^{n}_{ij}\right) \\ 
&= \dfrac{1}{2} \left( \dfrac{\partial b_i}{\partial x^j}+\dfrac{\partial b_j}{\partial x^i}\right)- c \Gamma^{n}_{ij}
\end{align*}
\begin{align*}
&= \dfrac{c}{2} \biggl\{  \left\langle \nabla_{\dfrac{\partial}{\partial x^j}} \dfrac{\partial}{\partial x^n}, \dfrac{\partial}{\partial x^i} \right\rangle+ \left\langle \dfrac{\partial}{\partial x^n},\nabla_{\dfrac{\partial}{\partial x^j}} \dfrac{\partial}{\partial x^i} \right\rangle\\
 &\ \ \ \ \ \ \ \ +   \left\langle \nabla_{\dfrac{\partial}{\partial x^i}} \dfrac{\partial}{\partial x^n}, \dfrac{\partial}{\partial x^j} \right\rangle+ \left\langle \dfrac{\partial}{\partial x^n},\nabla_{\dfrac{\partial}{\partial x^i}} \dfrac{\partial}{\partial x^j} \right\rangle \biggr\}- c \Gamma^{n}_{ij}\\
 &= \dfrac{c}{2} \left( \Gamma^{i}_{nj} + \Gamma^{j}_{ni} + 2\Gamma^{n}_{ij}\right) - c \Gamma^{n}_{ij} \\
 &= \dfrac{c}{2} \left( \Gamma^{i}_{nj} + \Gamma^{j}_{ni}\right) \\
 &=\dfrac{c}{4}\biggl\{ \bigg( -\left\langle \left[ v_n, v_j\right]_{\mathfrak m}, v_i \right\rangle + \left\langle \left[ v_i, v_n\right]_{\mathfrak m}, v_j \right\rangle + \left\langle \left[ v_i, v_j\right]_{\mathfrak m}, v_n \right\rangle  \bigg) \\
 & \; \; \; \; \; \; \; \ \ \ +\bigg( -\left\langle  \left[ v_n, v_i\right]_{\mathfrak m}, v_j \right\rangle  + \left\langle \left[ v_j, v_n\right]_{\mathfrak m}, v_i \right\rangle + \left\langle \left[ v_j, v_i\right]_{\mathfrak m}, v_n \right\rangle  \bigg) \biggr\} \\
 &= -\dfrac{c}{2}\left( \bigg<\left[ v_n, v_i\right] _{\mathfrak m}, v_j \bigg>+\bigg<\left[ v_n, v_j\right] _{\mathfrak m}, v_i \bigg>\right) .
\end{align*}
Therefore,
\begin{align*}
r_{_{00}} &= r_{ij} y^i y^j\\
&= -\dfrac{c}{2}\left( \bigg<\left[ v_n, v_i\right] _{\mathfrak m}, v_j \bigg>+\bigg<\left[ v_n, v_j\right] _{\mathfrak m}, v_i \bigg>\right) y^i y^j \\
&= -\dfrac{1}{2}\left( \bigg<\left[ c v_n, y^i v_i\right] _{\mathfrak m}, y^j v_j \bigg>+\bigg<\left[ c v_n, y^j v_j\right] _{\mathfrak m}, y^i v_i \bigg>\right) \\
&= -\dfrac{1}{2}\left( \bigg<\left[ v, y\right] _{\mathfrak m}, y \bigg>+\bigg<\left[ v, y\right] _{\mathfrak m}, y \bigg>\right)\\
&= - \bigg<\left[ v, y\right] _{\mathfrak m}, y \bigg>.
\end{align*}
Finally, substituting the values of $s_{_{0}}$ and $r_{_{00}}$ in the equation (\ref{e9}), we obtain the formula for $S$-curvature of  a homogeneous Finsler space with $(\alpha, \beta)$-metric.\\
The above discussion can be summarized in the following theorem:
\begin{theorem}{\label{thmS}}
	
	Let $F= \alpha \phi(s)$   be a $G$-invariant $(\alpha, \beta)$-metric on the reductive  homogeneous Finsler space $G/H$ with a decomposition of the Lie algebra $\mathfrak{g} = \mathfrak{h} +\mathfrak{m}$. Then the $S$-curvature  is given by 
	\begin{equation}{\label{e10}}
	S(H,y)=  \dfrac{\Phi}{2 \alpha \Delta^2}\left(\bigg<\left[ v, y\right] _{\mathfrak m}, y \bigg> + \alpha Q  \bigg<\left[v, y\right] _{\mathfrak m}, v \bigg>\right),
	\end{equation}
	where $v \in \mathfrak{m}$ corresponds to the 1-form $\beta$ and $\mathfrak{m}$ is identified with the tangent space $T_{H}\left( G/H\right) $ of $G/H$ at the origin $H$.
\end{theorem}
Here, we remark that the theorem \ref{thmS} is modified vesion of the theorem 2.1 given in (\cite{DH}). 	In reffered theorem, the authors have wrongly taken $c$, which should not be there.

Next, we compute the $S$-curvature of homogeneous Finsler space with infinite series $(\alpha, \beta)$-metric: $F= \dfrac{\beta^2}{\beta-\alpha}=\alpha \left(\dfrac{s^2}{s-1} \right)= \alpha \phi(s), \text{where} \ \ \phi(s)=  \dfrac{s^2}{s-1}.$ \\
For infinite series $(\alpha, \beta)$-metric , the identites given in the equation (\ref{eqS}) reduce to the following:
\begin{align*}
Q&= \dfrac{\phi'}{\phi-s \phi'}= 1-\dfrac{2}{s},\\
Q' &= \dfrac{2}{s^2},\\
Q'' &=-\dfrac{4}{s^3},\\
\Delta&= 1+sQ+\left(b^2-s^2\right)Q'\\
&=1+ s\left( \dfrac{s-2}{s}\right) + \left( b^2-s^2\right) \dfrac{2}{s^2}\\
&= \dfrac{s^3-3s^2+2b^2}{s^2}, \\
\Phi&= -\left( Q-sQ'\right) \left( n\Delta + 1 + sQ\right)- \left(b^2-s^2 \right) \left(1+sQ \right)Q''\\
&= -\left( 1-\dfrac{2}{s}- s\dfrac{2}{s^2}\right)\left\lbrace n \left( \dfrac{s^3-3s^2+2b^2}{s^2}\right) +1 + s\left( \dfrac{s-2}{s}\right)\right\rbrace \\
&\ \ \ \ \ \ -\left( b^2-s^2\right)\left\lbrace 1 + s\left( \dfrac{s-2}{s}\right)\right\rbrace \left( -\dfrac{4}{s^3}\right) \\
&= \dfrac{1}{s^3}\biggl\{ -\left( n+1\right)s^4 +\left(7n+1 \right)s^3 -12ns^2+2\left(2-n \right)b^2 s+4\left( 2n-1\right)b^2  \biggr\}.
\end{align*} 
Substituting the above values in the equation (\ref{e10}), we obtain the formula for   $S$-curvature of the homogeneous Finsler space with infinite series $(\alpha, \beta)$-metric in the form of  following theorem:
\begin{theorem}{\label{t1}}
	Let $F= \dfrac{\beta^2}{\beta-\alpha}$   be a $G$-invariant infinite series metric on the reductive  homogeneous Finsler space $G/H$ with a decomposition of the Lie algebra $\mathfrak{g} = \mathfrak{h} +\mathfrak{m}$. Then the $S$-curvature  is given by 
	\begin{equation}{\label{ee1}}
	S(H,y)= \left[\dfrac{\splitdfrac{ -\left( n+1\right)s^5 +\left(7n+1 \right)s^4 -12ns^3}{+2\left(2-n \right)b^2 s^2+4\left( 2n-1\right)b^2s }}{2\left( s^3-3s^2+2b^2\right)^2 } \right]\left(  \dfrac{1}{\alpha}\left\langle \left[ v, y\right] _{\mathfrak m}, y \right\rangle  + \left(1-\dfrac{2}{s} \right)   \left\langle \left[v, y\right] _{\mathfrak m}, v \right\rangle \right), 
	\end{equation}
	
	where $v \in \mathfrak{m}$ corresponds to the 1-form $\beta$ and $\mathfrak{m}$ is identified with the tangent space $T_{H}\left( G/H\right) $ of $G/H$ at the origin $H$.
\end{theorem}
Next, we compute the $S$-curvature of a homogeneous Finsler space with exponential metric : $F= \alpha  e^{\beta/\alpha}= \alpha e^s= \alpha \phi(s), \text{where} \ \ \phi(s)= e^s.$\\
For exponential metric, the identities given in the equation (\ref{eqS}) reduce to the following:
\begin{align*}
Q&= \dfrac{\phi'}{\phi-s \phi'}= \dfrac{1}{1-s},\\
Q' &= \dfrac{1}{\left( 1-s\right)^2 },\\
Q'' &=\dfrac{2}{\left( 1-s\right)^3 },
\end{align*}
\begin{align*}
\Delta&= 1+sQ+\left(b^2-s^2\right)Q' \\
&= 1+\dfrac{s}{1-s}+\dfrac{b^2-s^2}{\left( 1-s\right)^2}\\
&= \dfrac{1+b^2-s^2-s}{\left( 1-s\right)^2},\\
\Phi&= -\left( Q-sQ'\right) \left( n\Delta + 1 + sQ\right)- \left(b^2-s^2 \right) \left(1+sQ \right)Q''\\
&= -\left\lbrace \dfrac{1}{1-s}-\dfrac{s}{\left( 1-s\right)^2}\right\rbrace \left\lbrace \dfrac{n\left( 1+b^2-s^2-s\right)}{\left( 1-s\right)^2}+1+\dfrac{s}{\left( 1-s\right)}\right\rbrace \\
& \ \ \ \ -\left( b^2-s^2\right)\left\lbrace 1+\dfrac{s}{\left( 1-s\right)} \right\rbrace \dfrac{2}{\left( 1-s\right)^3}  \\
&= \dfrac{-\left\lbrace 2ns^3+ns^2-\left( 3+3n+2nb^2\right)s +\left( 2+n\right)b^2 +n+1\right\rbrace  }{\left( 1-s\right)^4}.
\end{align*} 
Substituting the above values in the equation (\ref{e10}), we obtain the formula for   $S$-curvature of the homogeneous Finsler space with exponential  metric. Hence we have the following theorem.
\begin{theorem}{\label{texp1}}
	Let $F= \alpha  e^{\beta/\alpha}$   be a $G$-invariant exponential metric on the reductive  homogeneous Finsler space $G/H$ with a decomposition of the Lie algebra $\mathfrak{g} = \mathfrak{h} +\mathfrak{m}$. Then the $S$-curvature is given by 
	\begin{equation}{\label{eeexp1}}
	S(H,y)= -\left[\dfrac{\splitdfrac{ 2ns^3+ns^2-\left( 3+3n+2nb^2\right)s}{+\left( 2+n\right)b^2 +n+1}}{2\left( 1+b^2-s^2-s\right)^2} \right]\left(  \dfrac{1}{\alpha}\left\langle \left[ v, y\right] _{\mathfrak m}, y \right\rangle  + \dfrac{1}{\left( 1-s\right) }   \left\langle \left[v, y\right] _{\mathfrak m}, v\right\rangle \right),
	\end{equation}
	where $v \in \mathfrak{m}$ corresponds to the 1-form $\beta$ and $\mathfrak{m}$ is identified with the tangent space $T_{H}\left( G/H\right) $ of $G/H$ at the origin $H$.
\end{theorem}

\begin{definition}
	An $n$-dimensional  Finsler space $(M,F)$ is said to have almost isotropic $S$-curvature if there exists a smooth function $c(x)$ on $M$ and a closed 1-form $\omega$ such that 
	$$ S(x,y)= (n+1) \bigg( c(x) F(y) + \omega(y)\bigg), \ \ x \in M, \   y\in T_x(M). $$
	In addition, if $\omega$ is zero, then   $(M,F)$ is said to have isotropic $S$-curvature.\\ 
	Also, if  $\omega$ is zero and $c(x)$ is constant, then $(M,F)$ is said to have  constant $S$-curvature.
\end{definition}
Next, we give an application of Theorem (\ref{t1}). 
\begin{theorem}
	Let $F= \dfrac{\beta^2}{\beta-\alpha}$   be a $G$-invariant infinite series metric on the reductive  homogeneous Finsler space $G/H$ with a decomposition of the Lie algebra $\mathfrak{g} = \mathfrak{h} +\mathfrak{m}$. Then $(G/H, F)$ has isotropic $S$-curvature if and only if  it has vanishing $S$-curvature.
\end{theorem}
\begin{proof}  We only need to prove the necessary part. Suppose $G/H$ has isotropic $S$-curvature, then we have 
	$$ S(x,y)= (n+1) c(x) F(y) , \ \ x \in G/H, \   y\in T_x(G/H).$$
	Letting $x=H$ and $y = v$ in the equation (\ref{ee1}), we get $c(H)=0$ and hence $S(H,y)=0 \ \forall \ y \in T_H(G/H). $ Further, since $F$ is a homogeneous metric, we have $S=0$ everywhere. \\
	Therefore, $G/H$ has vanishing $S$-curvature.
\end{proof}
Further, we give an application of Theorem (\ref{texp1}). 
\begin{theorem}
	Let $F= \alpha  e^{\beta/\alpha}$   be a $G$-invariant exponential metric on the reductive  homogeneous Finsler space $G/H$ with a decomposition of the Lie algebra $\mathfrak{g} = \mathfrak{h} +\mathfrak{m}$. Then $(G/H, F)$ has isotropic $S$-curvature if and only if  it has vanishing $S$-curvature.
\end{theorem}
\begin{proof} We only need to prove the necessary part. Suppose $G/H$ has isotropic $S$-curvature, then we have 
	$$ S(x,y)= (n+1) c(x) F(y) , \ \ x \in G/H, \   y\in T_x(G/H).$$
	Letting $x=H$ and $y = v$ in the equation (\ref{eeexp1}), we get $c(H)=0$ and hence $S(H,y)=0 \ \forall \ y \in T_H(G/H). $ Further, since $F$ is a homogeneous metric, we have $S=0$ everywhere. \\
	Therefore, $G/H$ has vanishing $S$-curvature.	
\end{proof} 

\section{Mean Berwald Curvature}
In this section, we find the mean Berwald curvature of the homogeneous Finsler space with infinite series metric and also for exponential metric. We first discuss the notion of the mean Berwald curvature (\cite{CSBOOK}) of a Finsler space (M, F). For this, let
$$ E_{ij}(x, y)= \dfrac{1}{2}\dfrac{\partial^2S(x, y)}{\partial y^i \partial y^j}.$$
The $E$-tensor is a family of symmetric forms $E_y \colon T_xM \times T_x M \longrightarrow \mathbb{R}$ defined by $E_y(u,v)= E_{ij}(x,y) u^iv^j,$ where $u=u^i \dfrac{\partial}{\partial x^i}, \ v=v^i \dfrac{\partial}{\partial x^i} \in T_x M, \ x \in M.$ Then $ E=\bigg\{ E_y \ \colon \ y \in TM \backslash \{0\}  \bigg\}  $ is called the $E$-curvature or the mean Berwald curvature.
First, we find it for a homogeneous Finsler space with infinite series metric.
To find it, we require the following computations:\\
At the origin, $a_{ij}= \delta^i_j$ and therefore $ y_{_i}=y^i.$
\begin{align*}
\alpha_{_{y^i}} &=\dfrac{y_{_i}}{\alpha},\\
\beta_{_{y^i}} &= b_i,\\
s_{_{y^i}} &= \dfrac{\partial}{\partial y^i}\left( \dfrac{\beta}{\alpha}\right) = \dfrac{b_i \alpha - s y_{_i}}{\alpha^2},\\
s_{_{y^i y^j}} &= \dfrac{\partial}{\partial y^j}\left( \dfrac{b_i \alpha - s y_{_i}}{\alpha^2}\right) \\
&= \dfrac{\alpha^2\left\lbrace b_i \dfrac{y_{_j}}{\alpha}- \left( \dfrac{b_j \alpha -s y_{_j}}{\alpha^2}\right)y_{_i} - s \delta^i_j \right\rbrace - \left( b_i \alpha -s y_{_i}\right)2 \alpha \dfrac{y_{_j}}{\alpha}  }{\alpha^4}\\
&= \dfrac{-\left( b_i y_{j}+b_j y_{i}\right) \alpha + 3s y_{i} y_j -\alpha^2 s \delta ^i_j }{\alpha ^4},\\
A &= \dfrac{\biggl\{ -\left( n+1\right)s^5 +\left(7n+1 \right)s^4 -12ns^3+2\left(2-n \right)b^2 s^2+4\left( 2n-1\right)b^2s  \biggr\}}{2\left( s^3-3s^2+2b^2\right)^2 },\\
\dfrac{\partial A}{\partial y^j}&= \left[\dfrac{\splitdfrac{\splitdfrac{\splitdfrac{\left( s^3-3s^2+2b^2\right)^2}{\left\lbrace -5\left( n+1\right)s^4 +4\left(7n+1 \right)s^3 -36ns^2+4\left(2-n \right)b^2 s+4\left( 2n-1\right)b^2 \right\rbrace }}{-\left( s^3-3s^2+2b^2\right)\left(6s^2-12s \right)}}{ \left\lbrace -\left( n+1\right)s^5 +\left(7n+1 \right)s^4 -12ns^3+2\left(2-n \right)b^2 s^2+4\left( 2n-1\right)b^2s\right\rbrace }}{2\left( s^3-3s^2+2b^2\right)^4}\right]s_{y^j}
\end{align*}
\begin{align*}
&=\left[\dfrac{\splitdfrac{(n+1)s^7+(-11n+1)s^6+36ns^5-2\left\lbrace (n+13)b^2+18n \right\rbrace s^4 }{+4(n+13)b^2s^3-36b^2s^2+8(2-n)b^4s+8(2n-1)b^4}}{2\left( s^3-3s^2+2b^2\right)^3}\right]s_{y^j},
\end{align*}
\begin{align*}
\dfrac{\partial^2 A}{\partial y^i \partial y^j} &=\dfrac{\partial}{\partial y^i}\left[\dfrac{\splitdfrac{(n+1)s^7+(-11n+1)s^6+36ns^5-2\left\lbrace (n+13)b^2+18n \right\rbrace s^4 }{+4(n+13)b^2s^3-36b^2s^2+8(2-n)b^4s+8(2n-1)b^4}}{2\left( s^3-3s^2+2b^2\right)^3}\right] s_{y^j} \\
&\ \ \ \ \ + \left[\dfrac{\splitdfrac{(n+1)s^7+(-11n+1)s^6+36ns^5-2\left\lbrace (n+13)b^2+18n \right\rbrace s^4 }{+4(n+13)b^2s^3-36b^2s^2+8(2-n)b^4s+8(2n-1)b^4}}{2\left( s^3-3s^2+2b^2\right)^3}\right]s_{y^i y^j}\\
&=\left[\dfrac{\splitdfrac{\splitdfrac{\splitdfrac{-2(n+1)s^9-6(-5n+1)s^8-144ns^7+24\left\lbrace (n+6)b^2+12n \right\rbrace s^6}{ -4\left\lbrace(37n+114)b^2+54n \right\rbrace s^5+36(20+11n)b^2s^4}}{+48\left\lbrace (-7+n)b^4-(n+9)b^2\right\rbrace s^3+48(13-5n)b^4s^2}}{+288(n-1)b^4s+16(2-n)b^6}}{2\left( s^3-3s^2+2b^2\right)^4}\right]s_{y^i} s_{y^j} \\
&\ \ \ \ \ + \left[\dfrac{\splitdfrac{(n+1)s^7+(-11n+1)s^6+36ns^5-2\left\lbrace (n+13)b^2+18n \right\rbrace s^4 }{+4(n+13)b^2s^3-36b^2s^2+8(2-n)b^4s+8(2n-1)b^4}}{2\left( s^3-3s^2+2b^2\right)^3}\right]s_{y^i y^j}.
\end{align*}
From the equation (\ref{ee1}),  $S$- curvature at the origin is given by 
$$ S(H,y) =  \dfrac{A}{\alpha} \left\langle \left[ v, y\right] _{\mathfrak m}, y \right\rangle + A\left(1-\dfrac{2}{s}\right)  \left\langle \left[v, y\right] _{\mathfrak m}, v\right\rangle.  $$

Further, we can write 
$$ S(H,y)= \mathrm{I}+ \mathrm{II}, $$ 
where $$ \mathrm{I}=\dfrac{A}{\alpha} \left\langle \left[ v, y\right] _{\mathfrak m}, y \right\rangle\ \ \text{and} \ \ \mathrm{II}= A\left(1-\dfrac{2}{s}\right)  \left\langle \left[v, y\right] _{\mathfrak m}, v\right\rangle  . $$
Therefore, the mean Berwald curvature is 
$$ \dfrac{1}{2}\dfrac{\partial^2 S}{\partial y^i \partial y^j}= \dfrac{1}{2}\left( \dfrac{\partial^2 \mathrm{I}}{\partial y^i \partial y^j}+ \dfrac{\partial^2 \mathrm{II}}{\partial y^i \partial y^j}\right), $$
where,
\begin{align*}
\dfrac{\partial \mathrm{I}}{\partial y^j} &= \dfrac{\partial}{\partial y^j}\left(\dfrac{A}{\alpha} \left\langle \left[ v, y\right] _{\mathfrak m}, y \right\rangle  \right) \\
&=\left( \dfrac{1}{\alpha}\dfrac{\partial A}{\partial y^j} - \dfrac{A}{\alpha^2} \dfrac{y_j}{\alpha}\right) \left\langle \left[ v, y\right] _{\mathfrak m}, y \right\rangle + \dfrac{A}{\alpha}\bigg( \left\langle \left[ v, v_j\right] _{\mathfrak m}, y \right\rangle + \left\langle \left[ v, y\right] _{\mathfrak m}, v_j \right\rangle  \bigg),  \\
\dfrac{\partial^2 \mathrm{I}}{\partial y^i \partial y^j} &= \dfrac{\partial}{\partial y^i}\left\lbrace \left( \dfrac{1}{\alpha}\dfrac{\partial A}{\partial y^j} - \dfrac{Ay_j}{\alpha^3} \right) \left\langle \left[ v, y\right] _{\mathfrak m}, y \right\rangle + \dfrac{A}{\alpha}\bigg( \left\langle \left[ v, v_j\right] _{\mathfrak m}, y \right\rangle + \left\langle \left[ v, y\right] _{\mathfrak m}, v_j \right\rangle  \bigg) \right\rbrace \\
&= \left( \dfrac{1}{\alpha}\dfrac{\partial^2 A}{\partial y^i \partial y^j}-\dfrac{1}{\alpha^2}\dfrac{y_i}{\alpha}\dfrac{\partial A}{\partial y^j}-\dfrac{y_j}{\alpha^3}\dfrac{\partial A}{\partial y^i}-\dfrac{A}{\alpha^3}\delta^j_{i}+\dfrac{3Ay_j}{\alpha^4}\dfrac{y_i}{\alpha}\right) \left\langle \left[ v, y\right] _{\mathfrak m}, y \right\rangle  \\
& \ \ \ \  + \left( \dfrac{1}{\alpha}\dfrac{\partial A}{\partial y^j} - \dfrac{A y_j} {\alpha^3}\right)\bigg( \left\langle \left[ v, v_i\right] _{\mathfrak m}, y \right\rangle +\left\langle \left[ v, y\right] _{\mathfrak m}, v_i \right\rangle  \bigg) \\
&\ \ \ \  +\left( \dfrac{1}{\alpha}\dfrac{\partial A}{\partial y^i}-\dfrac{A}{\alpha^2}\dfrac{y_i}{\alpha}\right) \bigg( \left\langle \left[ v, v_j\right] _{\mathfrak m}, y \right\rangle +\left\langle \left[ v, y\right] _{\mathfrak m}, v_j \right\rangle  \bigg) \\
& \ \ \ \ + \dfrac{A}{\alpha}\bigg( \left\langle \left[ v, v_j\right] _{\mathfrak m}, v_i \right\rangle +\left\langle \left[ v, v_i\right] _{\mathfrak m}, v_j \right\rangle  \bigg)    \\
&= \left( \dfrac{1}{\alpha}\dfrac{\partial^2 A}{\partial y^i \partial y^j}-\dfrac{y_i}{\alpha^3}\dfrac{\partial A}{\partial y^j}-\dfrac{y_j}{\alpha^3}\dfrac{\partial A}{\partial y^i}-\dfrac{A}{\alpha^3}\delta^j_{i}+\dfrac{3Ay_{i}y_j}{\alpha^5}\right) \left\langle \left[ v, y\right] _{\mathfrak m}, y \right\rangle  \\
& \ \ \ \  + \left( \dfrac{1}{\alpha}\dfrac{\partial A}{\partial y^j} - \dfrac{A y_j} {\alpha^3}\right)\bigg( \left\langle \left[ v, v_i\right] _{\mathfrak m}, y \right\rangle +\left\langle \left[ v, y\right] _{\mathfrak m}, v_i \right\rangle  \bigg) \\
&\ \ \ \  +\left( \dfrac{1}{\alpha}\dfrac{\partial A}{\partial y^i}-\dfrac{A y_i}{\alpha^3}\right) \bigg( \left\langle \left[ v, v_j\right] _{\mathfrak m}, y \right\rangle +\left\langle \left[ v, y\right] _{\mathfrak m}, v_j \right\rangle  \bigg) \\
& \ \ \ \ + \dfrac{A}{\alpha}\bigg( \left\langle \left[ v, v_j\right] _{\mathfrak m}, v_i \right\rangle +\left\langle \left[ v, v_i\right] _{\mathfrak m}, v_j \right\rangle  \bigg) ,  
\end{align*}
and 
\begin{align*}
\dfrac{\partial \mathrm{II}}{\partial y^j} &=\dfrac{\partial}{\partial y^j}\left( A\left( 1-\dfrac{2}{s} \right) \left\langle \left[v, y\right] _{\mathfrak m}, v\right\rangle \right)\\
&= \left\lbrace \left( 1-\dfrac{2}{s} \right)\dfrac{\partial A}{\partial y^j} + \dfrac{2As_{y^j}}{s^2}\right\rbrace \left\langle \left[v, y\right] _{\mathfrak m}, v\right\rangle + A\left( 1-\dfrac{2}{s} \right)\left\langle \left[v, v_j\right] _{\mathfrak m}, v\right\rangle .\\
\dfrac{\partial^2 \mathrm{II}}{\partial y^i \partial y^j }
&= \dfrac{\partial}{\partial y^i}\left[ \left\lbrace \left( 1-\dfrac{2}{s} \right)\dfrac{\partial A}{\partial y^j} + \dfrac{2As_{y^j}}{s^2}\right\rbrace \left\langle \left[v, y\right] _{\mathfrak m}, v\right\rangle + A\left( 1-\dfrac{2}{s} \right)\left\langle \left[v, v_j\right] _{\mathfrak m}, v\right\rangle\right] \\
&= \left\lbrace \left( 1-\dfrac{2}{s} \right)\dfrac{\partial^2 A}{\partial y^i \partial y^j }+ \dfrac{2s_{y^i}}{s^2}\dfrac{\partial A}{\partial y^j} +\dfrac{2s_{y^j}}{s^2}\dfrac{\partial A}{\partial y^i}-\dfrac{4As_{y^i}s_{y^j}}{s^3}+\dfrac{2A}{s^2}s_{y^i y^j}\right\rbrace \left\langle \left[v, y\right] _{\mathfrak m}, v\right\rangle \\
& \ \ \ \ \ +  \left\lbrace \left( 1-\dfrac{2}{s} \right)\dfrac{\partial A}{\partial y^j} + \dfrac{2As_{y^j}}{s^2}\right\rbrace \left\langle \left[v, v_i\right] _{\mathfrak m}, v\right\rangle + \left\lbrace \left( 1-\dfrac{2}{s} \right)\dfrac{\partial A}{\partial y^i} + \dfrac{2As_{y^i}}{s^2}\right\rbrace \left\langle \left[v, v_j\right] _{\mathfrak m}, v\right\rangle + 0.
\end{align*}
The above computations regarding  mean Berwald curvature of the homogeneous Finsler space with infinite series  $(\alpha, \beta)$-metric are summarized as follows:
\begin{theorem}
	Let $F= \dfrac{\beta^2}{\beta-\alpha}$   be a $G$-invariant infinite series $(\alpha, \beta)$-metric on the reductive  homogeneous Finsler space $G/H$ with a decomposition of the Lie algebra $\mathfrak{g} = \mathfrak{h} +\mathfrak{m}$. Then the mean Berwald curvature of the homogeneous Finsler space with infinite series metric  is given by 
	\begin{align*}
	E_{ij}(H,y)&= \dfrac{1}{2}\bigg[ \left( \dfrac{1}{\alpha}\dfrac{\partial^2 A}{\partial y^i \partial y^j}-\dfrac{y_i}{\alpha^3}\dfrac{\partial A}{\partial y^j}-\dfrac{y_j}{\alpha^3}\dfrac{\partial A}{\partial y^i}-\dfrac{A}{\alpha^3}\delta^j_{i}+\dfrac{3Ay_{i}y_j}{\alpha^5}\right) \left\langle \left[ v, y\right] _{\mathfrak m}, y \right\rangle  \\
	& \ \ \ \  + \left( \dfrac{1}{\alpha}\dfrac{\partial A}{\partial y^j} - \dfrac{A y_j} {\alpha^3}\right)\bigg( \left\langle \left[ v, v_i\right] _{\mathfrak m}, y \right\rangle +\left\langle \left[ v, y\right] _{\mathfrak m}, v_i \right\rangle  \bigg) \\
	&\ \ \ \  +\left( \dfrac{1}{\alpha}\dfrac{\partial A}{\partial y^i}-\dfrac{A y_i}{\alpha^3}\right) \bigg( \left\langle \left[ v, v_j\right] _{\mathfrak m}, y \right\rangle +\left\langle \left[ v, y\right] _{\mathfrak m}, v_j \right\rangle  \bigg) \\
	& \ \ \ \ + \dfrac{A}{\alpha}\bigg( \left\langle \left[ v, v_j\right] _{\mathfrak m}, v_i \right\rangle +\left\langle \left[ v, v_i\right] _{\mathfrak m}, v_j \right\rangle  \bigg)   \\
	&\ \ \ \  + \left\lbrace \left( 1-\dfrac{2}{s} \right)\dfrac{\partial^2 A}{\partial y^i \partial y^j }+ \dfrac{2s_{y^i}}{s^2}\dfrac{\partial A}{\partial y^j} +\dfrac{2s_{y^j}}{s^2}\dfrac{\partial A}{\partial y^i}-\dfrac{4As_{y^i}s_{y^j}}{s^3}+\dfrac{2A}{s^2}s_{y^i y^j}\right\rbrace \left\langle \left[v, y\right] _{\mathfrak m}, v\right\rangle \\
	& \ \ \ \  +  \left\lbrace \left( 1-\dfrac{2}{s} \right)\dfrac{\partial A}{\partial y^j} + \dfrac{2As_{y^j}}{s^2}\right\rbrace \left\langle \left[v, v_i\right] _{\mathfrak m}, v\right\rangle + \left\lbrace \left( 1-\dfrac{2}{s} \right)\dfrac{\partial A}{\partial y^i} + \dfrac{2As_{y^i}}{s^2}\right\rbrace \left\langle \left[v, v_j\right] _{\mathfrak m}, v\right\rangle 
	\bigg],
	\end{align*}
	where $v \in \mathfrak{m}$ corresponds to the 1-form $\beta$ and $\mathfrak{m}$ is identified with the tangent space $T_{H}\left( G/H\right) $ of $G/H$ at the origin $H$.
\end{theorem}
Next, we compute mean Berwald curvature for homogeneous Finsler space with exponential metric.
To find it, we require the following computations:\\
At the origin, $a_{ij}= \delta^i_j$ and therefore $ y_{_i}=y^i.$
\begin{align*}
B &= \dfrac{ 2ns^3+ns^2-\left( 3+3n+2nb^2\right)s +\left( 2+n\right)b^2 +n+1  }{2\left( -s^2-s+1+b^2\right) ^2}.\\
\dfrac{\partial B}{\partial y^j} &= \left[\dfrac{\splitdfrac{\left( -s^2-s+b^2+1\right)^2 \bigg\{ 6ns^2+2ns-\left( 3+3n+2nb^2\right) \bigg\}  +\left( -s^2-s+1+b^2\right) \left(4s+2 \right) }{\bigg\{ 2ns^3+ns^2-\left( 3+3n+2nb^2\right)s +\left( 2+n\right)b^2 +n+1 \bigg\}  } }{2\left( -s^2-s+1+b^2\right) ^4}\right] s_{y^j} \\
&= \left[ \dfrac{2ns^4-3(n+3)s^2+\left( 4\left(n+2 \right)b^2+3n+1 \right)s-\left(1+n+3nb^2-b^2+2nb^4 \right)  }{2\left( -s^2-s+1+b^2\right) ^3}\right]s_{y^j}.
\end{align*}
\begin{align*}
\dfrac{\partial^2 B}{\partial y^i \partial y^j}&=  \dfrac{\partial}{\partial y^i}\left[ \dfrac{2ns^4-3(n+3)s^2+\left( 4\left(n+2 \right)b^2+3n+1 \right)s-\left(1+n+3nb^2-b^2+2nb^4 \right)  }{2\left( -s^2-s+1+b^2\right) ^3}\right] s_{y^j}\\
&\ \ \ \ \ \ + \left[ \dfrac{2ns^4-3(n+3)s^2+\left( 4\left(n+2 \right)b^2+3n+1 \right)s-\left(1+n+3nb^2-b^2+2nb^4 \right)  }{2\left( -s^2-s+1+b^2\right) ^3}\right] s_{y^i y^j}\\
&=\left[ \dfrac{\splitdfrac{4ns^5-2ns^4-4\left( n-2nb^2+8\right)s^3 +\left\lbrace20\left( n+2\right)b^2 +6\left(2n-1 \right)  \right\rbrace s^2 }{+2\left(-6nb^4-8nb^2+2b^2-3n-11 \right)s +2\left\lbrace \left( -n+6\right)b^2+\left(-n+4 \right)b^4-1  \right\rbrace  }}{2\left( -s^2-s+1+b^2\right) ^4}\right] s_{y^i}s_{y^j}\\
& \ \ \ \ \ + \left[ \dfrac{2ns^4-3(n+3)s^2+\left( 4\left(n+2 \right)b^2+3n+1 \right)s-\left(1+n+3nb^2-b^2+2nb^4 \right)  }{2\left( -s^2-s+1+b^2\right) ^3}\right] s_{y^i y^j}.
\end{align*}
From  equation (\ref{eeexp1}),  $S$-curvature for a homogeneous Finsler space with exponential metric at the origin is given by 
$$ S(H,y) = -\left( \dfrac{B}{\alpha} \left\langle \left[ v, y\right] _{\mathfrak m}, y \right\rangle + \dfrac{B}{1-s} \left\langle \left[v, y\right] _{\mathfrak m}, v\right\rangle\right). $$  

Further, we can write 
$$ S(H,y)= -\left( \mathrm{III}+ \mathrm{IV}\right), $$ 
where $$ \mathrm{III}=\dfrac{B}{\alpha} \left\langle \left[ v, y\right] _{\mathfrak m}, y \right\rangle \ \ \text{and} \ \ \mathrm{IV}= \dfrac{B}{1-s} \left\langle \left[v, y\right] _{\mathfrak m}, v\right\rangle. $$
Therefore, the mean Berwald curvature of a homogeneous Finsler space with exponential metric is 
$$ \dfrac{1}{2}\dfrac{\partial^2 S}{\partial y^i \partial y^j}= -\dfrac{1}{2}\left( \dfrac{\partial^2 \mathrm{III}}{\partial y^i \partial y^j}+ \dfrac{\partial^2 \mathrm{IV}}{\partial y^i \partial y^j}\right), $$
where,
\begin{align*}
\dfrac{\partial \mathrm{III}}{\partial y^j} &= \dfrac{\partial}{\partial y^j}\left(\dfrac{B}{\alpha} \left\langle \left[ v, y\right] _{\mathfrak m}, y \right\rangle  \right) \\
&=\left( \dfrac{1}{\alpha}\dfrac{\partial B}{\partial y^j} - \dfrac{B}{\alpha^2} \dfrac{y_j}{\alpha}\right) \left\langle \left[ v, y\right] _{\mathfrak m}, y \right\rangle + \dfrac{B}{\alpha}\bigg( \left\langle \left[ v, v_j\right] _{\mathfrak m}, y \right\rangle + \left\langle \left[ v, y\right] _{\mathfrak m}, v_j \right\rangle  \bigg).\\
\dfrac{\partial^2 \mathrm{III}}{\partial y^i \partial y^j} &= \dfrac{\partial}{\partial y^i}\left\lbrace \left( \dfrac{1}{\alpha}\dfrac{\partial B}{\partial y^j} - \dfrac{By_j}{\alpha^3} \right) \left\langle \left[ v, y\right] _{\mathfrak m}, y \right\rangle + \dfrac{B}{\alpha}\bigg( \left\langle \left[ v, v_j\right] _{\mathfrak m}, y \right\rangle + \left\langle \left[ v, y\right] _{\mathfrak m}, v_j \right\rangle  \bigg) \right\rbrace \\
&= \left( \dfrac{1}{\alpha}\dfrac{\partial^2 B}{\partial y^i \partial y^j}-\dfrac{y_i}{\alpha^3}\dfrac{\partial B}{\partial y^j}-\dfrac{y_j}{\alpha^3}\dfrac{\partial B}{\partial y^i}-\dfrac{B}{\alpha^3}\delta^j_{i}+\dfrac{3By_{i}y_j}{\alpha^5}\right) \left\langle \left[ v, y\right] _{\mathfrak m}, y \right\rangle  \\
& \ \ \ \  + \left( \dfrac{1}{\alpha}\dfrac{\partial B}{\partial y^j} - \dfrac{B y_j} {\alpha^3}\right)\bigg( \left\langle \left[ v, v_i\right] _{\mathfrak m}, y \right\rangle +\left\langle \left[ v, y\right] _{\mathfrak m}, v_i \right\rangle  \bigg) \\
&\ \ \ \  +\left( \dfrac{1}{\alpha}\dfrac{\partial B}{\partial y^i}-\dfrac{B y_i}{\alpha^3}\right) \bigg( \left\langle \left[ v, v_j\right] _{\mathfrak m}, y \right\rangle +\left\langle \left[ v, y\right] _{\mathfrak m}, v_j \right\rangle  \bigg) \\
& \ \ \ \ + \dfrac{B}{\alpha}\bigg( \left\langle \left[ v, v_j\right] _{\mathfrak m}, v_i \right\rangle +\left\langle \left[ v, v_i\right] _{\mathfrak m}, v_j \right\rangle  \bigg), 
\end{align*}
and 
\begin{align*}
\dfrac{\partial \mathrm{IV}}{\partial y^j} &= \dfrac{\partial}{\partial y^j}\left( \dfrac{B}{1-s} \left\langle \left[v, y\right] _{\mathfrak m}, v\right\rangle \right) \\
&= \left( \dfrac{1}{\left( 1-s\right) }\dfrac{\partial B}{\partial y^j} + \dfrac{Bs_{y^j}}{\left( 1-s\right)^2}\right) \left\langle \left[v, y\right] _{\mathfrak m}, v\right\rangle + \dfrac{B}{\left( 1-s\right)}\left\langle \left[v, v_j\right] _{\mathfrak m}, v\right\rangle \\
\dfrac{\partial^2 \mathrm{IV}}{\partial y^i \partial y^j }
&= \dfrac{\partial}{\partial y^i}\left\lbrace \left( \dfrac{1}{\left( 1-s\right) }\dfrac{\partial B}{\partial y^j} + \dfrac{Bs_{y^j}}{\left( 1-s\right)^2}\right) \left\langle \left[v, y\right] _{\mathfrak m}, v\right\rangle + \dfrac{B}{\left( 1-s\right)}\left\langle \left[v, v_j\right] _{\mathfrak m}, v\right\rangle\right\rbrace \\
&= \left\lbrace  \dfrac{s_{y^i}}{\left( 1-s\right)^2} \dfrac{\partial B}{\partial y^j} + \dfrac{1}{\left( 1-s\right) }\dfrac{\partial^2 B}{\partial y^i \partial y^j} + \dfrac{2Bs_{y^i}s_{y^j}}{\left( 1-s\right)^3}+ \dfrac{s_{y^j}}{\left( 1-s\right)^2} \dfrac{\partial B}{\partial y^i}+ \dfrac{Bs_{y^i y^j}}{\left( 1-s\right)^2}\right\rbrace  \left\langle \left[v, y\right] _{\mathfrak m}, v\right\rangle \\
& \ \ \ \ + \left\lbrace \dfrac{1}{\left( 1-s\right) }\dfrac{\partial B}{\partial y^j} + \dfrac{Bs_{y^j}}{\left( 1-s\right)^2}\right\rbrace  \left\langle \left[v, v_i\right] _{\mathfrak m}, v\right\rangle + \left\lbrace  \dfrac{1}{\left( 1-s\right) }\dfrac{\partial B}{\partial y^i} + \dfrac{Bs_{y^i}}{\left( 1-s\right)^2}\right\rbrace  \left\langle \left[v, v_j\right] _{\mathfrak m}, v\right\rangle + 0. 
\end{align*}
The whole above discussion regarding mean Berwald curvature of the homogeneous Finsler space with exponential metric can be summarized in the following theorem.
\begin{theorem}
	Let $F= \alpha  e^{\beta/\alpha}$   be a $G$-invariant exponential metric on the reductive  homogeneous Finsler space $G/H$ with a decomosition of the Lie algebra $\mathfrak{g} = \mathfrak{h} +\mathfrak{m}$. Then the mean Berwald curvature of the homogeneous Finsler space with exponential metric is given by 
	\begin{align*}
	E_{ij}(H,y)&= -\dfrac{1}{2}\bigg[\left( \dfrac{1}{\alpha}\dfrac{\partial^2 B}{\partial y^i \partial y^j}-\dfrac{y_i}{\alpha^3}\dfrac{\partial B}{\partial y^j}-\dfrac{y_j}{\alpha^3}\dfrac{\partial B}{\partial y^i}-\dfrac{B}{\alpha^3}\delta^j_{i}+\dfrac{3By_{i}y_j}{\alpha^5}\right) \left\langle \left[ v, y\right] _{\mathfrak m}, y \right\rangle  \\
	& \ \ \ \  \ + \left( \dfrac{1}{\alpha}\dfrac{\partial B}{\partial y^j} - \dfrac{B y_j} {\alpha^3}\right)\bigg( \left\langle \left[ v, v_i\right] _{\mathfrak m}, y \right\rangle +\left\langle \left[ v, y\right] _{\mathfrak m}, v_i \right\rangle  \bigg) \\
	&\ \ \ \  \ +\left( \dfrac{1}{\alpha}\dfrac{\partial B}{\partial y^i}-\dfrac{B y_i}{\alpha^3}\right) \bigg( \left\langle \left[ v, v_j\right] _{\mathfrak m}, y \right\rangle +\left\langle \left[ v, y\right] _{\mathfrak m}, v_j \right\rangle  \bigg) \\
	& \ \ \ \ \ + \dfrac{B}{\alpha}\bigg( \left\langle \left[ v, v_j\right] _{\mathfrak m}, v_i \right\rangle +\left\langle \left[ v, v_i\right] _{\mathfrak m}, v_j \right\rangle  \bigg)\\
	&\ \ \ \ \ + \left\lbrace  \dfrac{s_{y^i}}{\left( 1-s\right)^2} \dfrac{\partial B}{\partial y^j} + \dfrac{1}{\left( 1-s\right) }\dfrac{\partial^2 B}{\partial y^i \partial y^j} + \dfrac{2Bs_{y^i}s_{y^j}}{\left( 1-s\right)^3}+ \dfrac{s_{y^j}}{\left( 1-s\right)^2} \dfrac{\partial B}{\partial y^i}+ \dfrac{Bs_{y^i y^j}}{\left( 1-s\right)^2}\right\rbrace  \left\langle \left[v, y\right] _{\mathfrak m}, v\right\rangle \\
	& \ \ \ \ \ + \left\lbrace \dfrac{1}{\left( 1-s\right) }\dfrac{\partial B}{\partial y^j} + \dfrac{Bs_{y^j}}{\left( 1-s\right)^2}\right\rbrace  \left\langle \left[v, v_i\right] _{\mathfrak m}, v\right\rangle + \left\lbrace  \dfrac{1}{\left( 1-s\right) }\dfrac{\partial B}{\partial y^i} + \dfrac{Bs_{y^i}}{\left( 1-s\right)^2}\right\rbrace  \left\langle \left[v, v_j\right] _{\mathfrak m}, v\right\rangle \bigg],
	\end{align*}
	where $v \in \mathfrak{m}$ corresponds to the 1-form $\beta$ and $\mathfrak{m}$ is identified with the tangent space $T_{H}\left( G/H\right) $ of $G/H$ at the origin $H$.
\end{theorem}

	\end{document}